\documentclass[fontsize=11pt, headings=normal, headinclude=true, paper=a4, pagesize, cleardoublepage=current, BCOR=10mm, fleqn, numbers=noenddot]{scrartcl}
\pdfoutput=1
\usepackage{lmodern}

\setkomafont{sectioning}{\normalcolor\bfseries}
\recalctypearea 
\setcapindent{0pt}

  \usepackage[english]{babel}  
  \usepackage[numbers,square]{natbib}
  \usepackage{mathrsfs}
  \usepackage{amsmath}
  \usepackage{amssymb}
  \usepackage{amsthm}
  \usepackage{enumitem}
  \usepackage[utf8]{inputenc}    
  \usepackage[T1]{fontenc}         
  \usepackage{graphicx}            
  \usepackage{color}		
  \usepackage[english=british]{csquotes} 
  
\newtheorem{thm}{Theorem}[section]
\newtheorem{prop}[thm]{Proposition}
\newtheorem{lem}[thm]{Lemma}

\newtheorem*{thm*}{Theorem}

\theoremstyle{definition}
\newtheorem{definition}[thm]{Definition}

\newtheorem{rem}[thm]{Remark}

\renewcommand{\phi}{\varphi}
\newcommand{\eps}{\varepsilon}

\newcommand{\ud}{\mathrm{d}}
\newcommand{\uod}{{\mathrm{od}}}
\newcommand{\ue}{\mathrm{e}}
\newcommand{\ui}{\mathrm{i}}
\newcommand{\R}{\mathbb{R}}
\newcommand{\C}{\mathbb{C}}
\newcommand{\N}{\mathbb{N}}

\newcommand{\norm}[1]{\ensuremath{\left\lVert #1 \right\rVert}}

\newcommand{\hilb}{\mathscr{H}}
\newcommand{\uppar}[1]{\ensuremath{^{(#1)}}}
  
\DeclareMathOperator{\ran}{ran}

\title{The resolvent of the Nelson Hamiltonian improves positivity}

\author{Jonas Lampart\thanks{Ludwig-Maximilians-Universit\"at, Mathematisches Institut, Theresienstr.\ 39, {80333} M\"unchen, Germany.
}
\thanks{CNRS \& LICB (UMR 6303), Université de Bourgogne Franche-Comté, 9 Av. A. Savary, 21078 Dijon Cedex, France.
\texttt{jonas.lampart@u-bourgogne.fr}}
}

\begin{document}

\maketitle

\begin{abstract}
We give a new proof that the resolvent of the renormalised Nelson Hamiltonian at fixed total momentum $P$ improves positivity in the (momentum) Fock-representation, for every $P$.
The argument is based on an explicit representation of the renormalised operator and its domain using interior boundary conditions, which allows us to avoid the intermediate steps of regularisation and renormalisation used in other proofs of this result.
\end{abstract}

\section{Introduction} 

An operator on a Hilbert space is said to preserve positivity if it leaves a cone of ``positive elements'', e.g. functions that are point-wise non-negative, invariant. It is said to improve positivity if it maps any non-zero positive element to a strictly positive element, meaning that the scalar product with any positive element is strictly positive.
This property has important consequences for the spectral theory of the operator. For example, if a self-adjoint bounded operator improves positivity and has a maximal eigenvalue, then this eigenvalue is simple, with a strictly positive eigenfunction, by the Perron-Frobenius-Faris theorem~\cite{faris1972}.
This method plays an important role in the spectral analysis of Hamiltonians from quantum field theory (QFT)~\cite{GlJa1970, gross1972, faris1972, frohlich1974, bach1998a, moller2005, DaHi19}. 

An important example is the Nelson model for the interaction of a non-relativistic particle with a bosonic field. 
At fixed total momentum $P$ this model is described by a self-adjoint  Hamiltonian acting on the symmetric Fock space
\begin{equation}
 \hilb_P:=\Gamma(L^2(\R^3))=\bigoplus_{n=0}^\infty \big(L^2(\R^3)\big)^{\otimes_{\mathrm{sym}} n}.
\end{equation}
The Hamiltonian has the formal expression
\begin{equation}\label{eq:H(P) formal}
 H(P)=\big(P-\ud \Gamma(k)\big)^2 + \ud \Gamma(\omega) + a(v) + a^*(v),
\end{equation}
 where $\ud \Gamma(k)$ is the field momentum, acting on each factor $L^2(\R^3)$ as multiplication by the respective variable, $\omega$ denotes the operator of multiplication by the dispersion relation $\omega(k) =\sqrt{k^2 + m^2}$ (where $m\geq 0$ is the boson mass) and $v(k)=g\omega(k)^{-1/2}$, $g\in \R$, is the form-factor of the interaction. This expression cannot be interpreted as a sum of densely defined operators on $\hilb_P$ since $v\notin L^2(\R^3)$.
 However, it is possible to define a self-adjoint renormalised Hamiltonian corresponding to the formal expression. This was constructed for the translation-invariant model by Nelson~\cite{nelson1964}, and by Cannon~\cite{Can71} for the model at fixed momentum (see also~\cite{GrWu18, DaHi19} for a recent exposition and refinements). This operator is obtained as the limit $\Lambda \to \infty$ (in norm-resolvent sense) of the operators $H_\Lambda(P)-E_\Lambda\to H_\mathrm{ren}(P)$ with ultraviolet (UV) cutoff, where  $v$ is replaced by $v_\Lambda(k)=v(k)1(|k|\leq \Lambda)$ and $E_\Lambda \in \R$ are appropriately chosen numbers. 

Fr\"ohlich~\cite{frohlich1973, frohlich1974}, and M{\o}ller~\cite{moller2005}, showed for the Nelson model with UV cutoff that the resolvent of the Hamiltonian as well as the generated semigroup improve positivity, and used this to prove that the ground state of the Hamiltonian $H_\Lambda(P)$ (which exists for $m>0$ and small $|P|$) is simple.
In~\cite{frohlich1973, frohlich1974} it was also announced that the same positivity property holds for the renormalised operator. However, a complete proof was given only recently by
Miyao~\cite{miyao2019, miyao2020}, who showed that the semigroup generated by the renormalised Hamiltonian improves positivity for every $P$, which  implies the same property for the resolvent. 
 This was then used by Dam and Hinrichs~\cite{DaHi19}, who proved non-existence of the ground state for $m=0$.
Positivity for $P=0$ in the path-integral representation had been shown earlier by Gross~\cite{gross1972} (for the model with cutoff), and by Matte and M{\o}ller~\cite{MaMo18} for the renormalised Hamiltonian. 
The difficulty in proving these results is, roughly speaking, that $H_\mathrm{ren}(P)$ is defined as a limit, and in this limit positive quantities could converge to zero.

In this article we give a new proof that the resolvent of $H_\mathrm{ren}(P)$ improves positivity.
We use a representation of the operator and its domain in terms of generalised boundary conditions, called interior boundary conditions, which allows us to work only with the renormalised operator and avoid approximation by operators with cutoff and the difficulties this entails.
The method of proof should be well suited for generalisations to other Hamiltonians, in particular the general Nelson-type models treated in~\cite{LaSch19,IBCrelat} (see Remark~\ref{rem:general}) and the Bogoliubov-Fröhlich Hamiltonian, whose renormalisation was recently achieved using the approach of interior boundary conditions~\cite{La19, La20}.

\section{Results}

In order to state our main results, we will first need to introduce some notation and
give the precise definition of the Hamiltonian by interior boundary conditions.
This approach to the UV problem was proposed by Teufel and Tumulka~\cite{TeTu16, TeTu20}. It was applied to the massive ($m>0$) Nelson model by Schmidt and the author~\cite{LaSch19}, and generalised to the massless case ($m=0$) by Schmidt~\cite{IBCmassless}. We reprove the self-adjointness of the Hamiltonian in Section~\ref{sect:sa} and then state the main result in Section~\ref{sect:res-pos}.

\subsection{Notation and defintion of the Hamiltonian}\label{sect:not}

For normed spaces $X,Y$ we denote by $\mathscr{B}(X,Y)$ the normed space of bounded linear operators from $X$ to $Y$, and write $\mathscr{B}(X):=\mathscr{B}(X,X)$.
For a densely defined operator $L, D(L)$ on $X$ we view $D(L)$ as a normed space with the graph norm $\|x\|_{D(L)}:=\|x\|_X + \|Lx\|_X$. 

For a unitary $U\in \mathscr{B}(L^2(\R^3))$ we define $\Gamma(U)$ as the induced unitary on $\hilb_P$ acting on $\hilb_P\uppar{n}:=\big(L^2(\R^3)\big)^{\otimes_{\mathrm{sym}} n}$ as $U^{\otimes n}$.
For a self-adjoint operator $L$ we denote by $\ud \Gamma (L)$ the self-adjoint generator of $\Gamma(\ue^{-\ui t L})$.
We will denote by $C$ various positive constants, whose value may change from one equation to another.

Let 
\begin{equation}
L_P:= (P-\ud \Gamma(k))^2 + \ud \Gamma(\omega)
\end{equation}
with $\omega(k)=\sqrt{k^2 + m^2}$, $m\geq0$.
This operator leaves the particle-number invariant and acts on an $n$-particle wavefunction as multiplication  by the non-negative function of $K=(k_1, \dots, k_n)$
\begin{equation}\label{eq:L_PK}
 L_P(K):=\Big(P-\sum_{j=1}^n k_j\Big)^2 + \sum_{j=1}^n \omega(k_j).
\end{equation}
It is thus self-adjoint on its maximal domain $D(L_P)\subset \hilb_P$ and non-negative.
For $\lambda>0$ we set, with $v(k)=g \omega(k)^{-1/2}$, $g\in \R$,
\begin{equation}
 G_\lambda^*:=-a(v)(L_P+\lambda)^{-1},
\end{equation}
where $a(v)$ is the annihilation operator that acts on (a dense subspace of) $\hilb_P\uppar{n+1}$ as
\begin{equation}
 \Big(a(v)\psi\uppar{n+1}\Big)(K)=\sqrt{n+1} \int_{\R^3} v(\xi) \psi\uppar{n+1}(K, \xi) \ud \xi.
\end{equation}
The operator $G_\lambda^*$ is bounded on $\hilb_P$ (see Lemma~\ref{lem:G}).
The action of the adjoint, which is thus also bounded, is given on $\hilb_P\uppar{n}$ by
\begin{align}
 \Big(G_\lambda \psi\uppar{n}\Big)(K)
 &=-\frac{1}{\sqrt{n+1}}\sum_{j=1}^{n+1} \frac{v(k_j) \psi \uppar n(\hat{K}_j)}{L_P(K)+\lambda},
 \label{eq:G n}
 \end{align}
 where $\hat K_j \in \R^{3(n-1)}$ denotes the vector $K$ with the entry $k_j$ removed.
 
The domain of the Hamiltonian at total momentum $P$ is 
\begin{equation}
 D(H_P)= \Big\{\psi \in \hilb_P : (1-G_\lambda)\psi \in D(L_P) \Big\}.
\end{equation}
The condition $(1-G_\lambda)\psi \in D(L_P)$ is the interior boundary condition that encodes the behaviour of $\psi$ for $k\to \infty$. Note that $\ran (G_\lambda-G_\mu)\subset D(L_P)$, so this condition is independent of $\lambda>0$.

The action of $H_P$ on its domain can be expressed as
\begin{equation}\label{eq:H lambda}
 H_P=(1-G_\lambda)^*(L_P +\lambda)(1-G_\lambda) + T_\lambda - \lambda,
\end{equation}
with the operator $T_\lambda= T_{\ud,\lambda} + T_{\mathrm{od},\lambda}$,  whose ``diagonal part'' acts on $\hilb_P\uppar{n}$, $n\in \N_0$, as the operator of multiplication by
\begin{align}\label{eq:T_d def}
 T_{\ud,\lambda}(K) %
 = &   \int_{\R^3} |v(\xi)|^2 \left(\frac{1}{\xi^2 +\omega(\xi)} - \frac{1}{L_P(K, \xi)+ \lambda} \right) \ud \xi
\end{align}
and the off-diagonal part acts as the integral operator (for $n>0$)
\begin{align}\label{eq:T_od def}
 \left(T_{\mathrm{od},\lambda}\psi\uppar{n}\right)(K)
 = -  \sum_{j=1}^{n} \int_{\R^3}  \frac{\overline{v(\xi)} v(k_{j}) \psi \uppar n(\hat{K}_j,\xi )}{L_P(K,\xi) + \lambda} \, \ud \xi.
\end{align}

To understand the connection of the Hamiltonian~\eqref{eq:H lambda} and the formal expression~\eqref{eq:H(P) formal}, we expand the former into a sum of terms that are individually not elements of $\hilb_P$, but of $D(L_P^{-1})$ (the completion of $\hilb_P$ under the norm $\|\psi\|=\|(L_P+1)^{-1}\psi\|_{\hilb_P}$).
In this sense, we have for $\psi\in D(H_P)$
\begin{align}\label{eq:H dist}
 H_P \psi 
 %
 %
 &=L_P\psi + a^*(v)\psi + A \psi,
\end{align}
where $A$ is an extension of $a(v)$, as we explain below. This shows that $H_P$ is independent of $\lambda$ and essentially acts as the formal expression~\eqref{eq:H(P) formal}, up to the choice of extension $A\supset a(v)$.

We now derive~\eqref{eq:H dist}. First note that, by boundedness of $G_\lambda$, $a(v)\in \mathscr{B}(D(L_P), \hilb_P)$, so we set $D(a(v))=D(L_P)$.
Let $D(T_\lambda):=D(L_P^\eps + \ud \Gamma(\omega)^{1/2} )\subset \hilb_P$ be the domain of $T_\lambda$ (for appropriate $\eps>0$, see Equation~\eqref{eq:T bound} below). Consider the subspace 
\begin{equation}
 D(A)_\lambda:=D(L_P) \oplus G_\lambda D(T_\lambda) \subset \hilb_P.
\end{equation}
Since, by its definition, $D(T_\lambda)$ is independent of $\lambda$, and $G_\lambda - G_\mu\in \mathscr{B}(\hilb_P, D(L_P))$ (as one easily checks using the resolvent identity), we have $D(A)_\lambda=D(A)_\mu$ for all  $\mu, \lambda>0$. We may thus drop the subscript $\lambda$, and we have $D(a(v))\subset D(A)$. We then set
\begin{equation}
 \begin{aligned}
  A:D(A) &\to \hilb_P \\
  A(\psi ,  G_\lambda\phi)&= a(v)\psi + T_\lambda\phi.
 \end{aligned}
\end{equation}
The operator $A$ is well defined and independent of $\lambda$ since 
\begin{equation}
 T_\lambda - T_\mu =  (\lambda - \mu)G_\lambda^*G_\mu\phi =a(v)(G_\lambda-G_\mu),
\end{equation}
which also follows from the resolvent identity.
For $\psi \in D(H_P)$, i.e. with $(1-G_\lambda)\psi \in D(L_P)$, we then have (see the proof of Theorem~\ref{thm:sa} for the justification of $\psi \in D(T_\lambda)$)
\begin{equation}
 A \psi = A(1-G_\lambda)\psi + AG_\lambda\psi = A\Big((1-G_\lambda)\psi ,G_\lambda \psi \Big) = a(v) (1-G_\lambda)\psi + T_\lambda\psi.
\end{equation}
 This implies~\eqref{eq:H dist}, since by definition of $G_\lambda$
\begin{align}
 H_P \psi &= (L_P +\lambda)(1-G_\lambda)\psi + a(v)(1-G_\lambda)\psi + T_\lambda\psi - \lambda\psi \notag \\
 %
 %
 &=L_P\psi + a^*(v)\psi + A \psi. \notag
\end{align}

\subsection{Self-adjointness of the Hamiltonian}\label{sect:sa}

It was proved in~\cite{LaSch19, IBCmassless} that the translation-invariant Nelson Hamiltonian (which is unitarily equivalent to the direct integral of the fibre-operators $H_P$~\cite{Can71}) is self-adjoint, bounded from below and equals the renormalised operator constructed in~\cite{nelson1964, Can71, GrWu18}. 
For the convenience of the reader, we reprove this result for the fibre Hamiltonian $H_P$, with key Lemmas provided in Appendix~\ref{app:ibc}. 

Concerning positvity, the important point is that the method of proof also yields a formula for $(H_P+\mu)^{-1}$ for large enough $\mu$, since it uses the Kato-Rellich theorem (see also~\cite{posi20} for alternative representations of the resolvent). In order to show positivity of the resolvent we will basically follow this proof with some modifications that give control over the sign of certain terms. In particular, the arguments of Section~\ref{sect:pos} also imply self-adjointness of $H_P$ (in a slightly different representation). Our result can thus almost be viewed as a corollary to the proof of self-adjointness, whereas the proof via renormalisation~\cite{miyao2019, miyao2020} requires a much finer analysis of the behaviour as $\Lambda\to \infty$. 
\begin{thm}\label{thm:sa}
The operator $H_P$ is self-adjoint on $D(H_P)$, bounded from below, and $H_P=H_\mathrm{ren}(P)$. 
\end{thm}

\begin{proof} 
 First, $(1-G_\lambda)$ has a bounded inverse (see Lemma~\ref{lem:G}), so $D(H_P)$ is dense. Furthermore, for $\lambda>0$
\begin{equation}\label{eq:H_0}
 (1-G_\lambda^*)(L_P + \lambda) (1-G_\lambda)
\end{equation}
is invertible and thus self-adjoint on $D(H_P)$ and non-negative. 

Self-adjointness of $H_P$ will follow from the Kato-Rellich theorem if we can show that $T_\lambda$ is bounded relative to this operator with infinitesimal bound. 
The key is that, by Lemma~\ref{lem:T_d} and Lemma~\ref{lem:T_od}, we have for any $\eps>0$
\begin{equation}\label{eq:T bound}
 \|T_\lambda  \psi\|_{\hilb_P} \leq C \Big( \|(L_P+\lambda)^{\eps}\psi\|_{\hilb_P} + \| \ud\Gamma(\omega)^{1/2} \psi\|_{\hilb_P}\Big).
\end{equation}
This implies that $T_\lambda(1-G_\lambda)$ is infinitesimally bounded relative to~\eqref{eq:H_0}, by boundedness of $(1-G_\lambda^*)^{-1}$. 
Taking $\eps<1/4$ it follows from Lemma~\ref{lem:G}\ref{lem:G L-bound} and Lemma~\ref{lem:G}\ref{lem:G O-bound}
that $T_\lambda G_\lambda$ is a bounded operator on $D(\ud \Gamma(\omega)^{1/2})$.
Since $(1-G_\lambda)$ is an isomorphism on $D(\ud \Gamma(\omega)^{1/2})$ by Lemma~\ref{lem:G}\ref{lem:G inv}, the operator $T_\lambda G_\lambda$ is bounded relative to $(1-G_\lambda)^*\ud \Gamma(\omega)^{1/2}(1-G_\lambda)$ and thus infinitesimally  bounded relative to~\eqref{eq:H_0}. This implies self-adjointness and the lower bound.

The equality $H_P=H_\mathrm{ren}(P)$ can be proved by showing that $H_P$ is the norm-resolvent limit of $H_\Lambda(P)-E_\Lambda$, where 
\begin{equation}
 H_\Lambda(P)= L_P + a(v_\Lambda) + a^*(v_\Lambda),
\end{equation}
with $v_\Lambda(k)=g\omega(k)^{-1/2} 1(|k|\leq \Lambda)$, and
\begin{equation}
 E_\Lambda = - \langle v_\Lambda, (k^2+\omega(k))^{-1} v_\Lambda \rangle_{L^2(\R^3)}.
\end{equation}
To see this, define $G_{z,\Lambda}:=-(L_P+z)^{-1}a^*(v_\Lambda)$, for $z\in \C\setminus\R_-$. Then rewrite 
\begin{equation}
 H_\Lambda(P) +z = (1-G_{\bar z,\Lambda})^*(L_P + z)(1-G_{z,\Lambda}) - a(v_\Lambda)(L_P + z)^{-1}a^*(v_\Lambda),
\end{equation}
and note that
\begin{equation*}
 - a(v_\Lambda)(L_P + z)^{-1}a^*(v_\Lambda) - E_\Lambda = T_{\ud, z, \Lambda} + T_{\uod, z, \Lambda}=:T_{z, \Lambda}, 
\end{equation*}
with $T_{\ud, z, \Lambda}$, $T_{\uod, z, \Lambda}$ defined as in~\eqref{eq:T_d def},~\eqref{eq:T_od def} with $v$ replaced by $v_\Lambda$ (and $\lambda$ by $z$).
Writing out the resolvent difference, we obtain, for $z\in \C\setminus \R$ and $\mu>0$,
\begin{align}
 (H_P &+ z)^{-1} - (H_\Lambda(P) -E_\Lambda+ z)^{-1} \\
 = &(H_\Lambda(P) -E_\Lambda+ z)^{-1}\big(H_\Lambda(P) - H_P\big)(H_P + z)^{-1} \notag \\
 = &(H_\Lambda(P) -E_\Lambda + z)^{-1}(1-G_{\mu, \Lambda})^*(L_P+\mu) (G_\mu-G_{\mu,\Lambda})(H_P + z)^{-1}\notag \\
 &+ (H_\Lambda(P) -E_\Lambda+ z)^{-1}(G_\mu-G_{\mu, \Lambda})^*(L_P+\mu) (1-G_\mu)(H_P + z)^{-1}\notag \\
 &+ (H_\Lambda(P) -E_\Lambda+ z)^{-1}(T_{\mu, \Lambda}- T_\mu)(H_P + z)^{-1}. \notag
\end{align}
The convergence to zero of this expression in norm then follows from
\begin{align}
G_{\mu, \Lambda} &\stackrel{\Lambda \to \infty}{\rightarrow} G_\mu \text{ in } \mathscr{B}(\hilb_P) \\
 T_{\mu,\Lambda} &\stackrel{\Lambda \to \infty}{\rightarrow} T_\mu \text{ in } \mathscr{B}(D(T_\mu), \hilb_P),
\end{align}
and the fact that $(H_\Lambda(P) + z)^{-1}(1-G_{\mu,\Lambda})^*(L_P+\mu)$ is bounded uniformly in $\Lambda$, which can be proved along the lines of Lemmas~\ref{lem:G},~\ref{lem:T_d},~\ref{lem:T_od} (see~\cite{IBCmassless, IBCrelat} for details).
\end{proof}

\begin{rem}\label{rem:general}
 The results of~\cite{LaSch19, IBCmassless} apply to general Nelson-type models in $d\leq 3 $ dimensions with $v, \omega$ satisfying $|v(k)|\leq |g| |k|^{-\alpha}$, $\omega(k)\geq (m^2+k^2)^{\beta/2}$ with appropriate conditions on $\alpha, \beta$ (see also~\cite{IBCrelat} for models with relativistic particles). Our method also works for these more general models under the additional hypothesis that $\beta/3 > d-2-2\alpha$ (which ensures that $T_{\ud, \lambda}G_\lambda$ is bounded).
\end{rem}

\subsection{Positivity of the resolvent of the Nelson Hamiltonian}\label{sect:res-pos}

In this section we formulate and prove our main result.  
\begin{definition}
 We define the cone of positive elements $ \mathcal{C}_+ \subset \hilb_P$ by
 \begin{equation}
 \mathcal{C}_+:=\{ \psi \in \hilb_P| \forall n\in \N_0 : \psi\uppar{n} \geq 0 \text{ almost everywhere} \}.
\end{equation}
We write that $\psi\geq 0$ if $\psi\in \mathcal{C}_+$ and $\psi>0$ if $\langle \psi, \phi \rangle > 0$ for all non-zero $\phi\geq 0$.
An operator $A\in \mathscr{B}(\hilb_P)$ preserves positivity if $A\mathcal{C}_+\subset \mathcal{C}_+$, and improves positivity if $A\psi>0$ for all $\psi \in \mathcal{C}_+\setminus\{0\}$.
\end{definition}

We will show that the resolvent of $H_P$ improves positivity if $g<0$, that is if $v(k)$ is strictly negative.

\begin{thm}\label{thm:H_P pos}
 Let $P\in \R^3$, $m\geq 0$ and $g<0$. Then for all $\lambda>-\inf\sigma(H_P)$ the resolvent  $(H_P+\lambda)^{-1}$ improves positivity with respect to  $\mathcal{C}_+$. 
\end{thm}

If instead $g>0$ and $v$ is positive all our results hold with positivity defined by the cone
\begin{equation}
 \mathcal{C}_-:=\{  \psi \in \hilb_P| \forall n\in \N_0 : (-1)^n\psi\uppar{n} \geq 0 \text{ almost everywhere} \},
\end{equation}
since the operators with interactions $v$ and $-v$ are unitarily equivalent via $U=\Gamma(-1)$.

To see why Theorem~\ref{thm:H_P pos} might hold, first note that $(L_P+\lambda)^{-1}$ preserves positivity since it is obtained from a multiplication operator by a positive function.
Using the sign of $v$, one easily sees that $G_\lambda$, $(1-G_\lambda)^{-1}=\sum_{j=0}^\infty G_\lambda^j$ (for $\lambda \gg 1$) and their adjoints preserve positivity by  inspection of the  formula~\eqref{eq:G n}.  
With this, the inverse of
\begin{equation}
 (1-G_\lambda)^*(L_P + \lambda)(1-G_\lambda)
\end{equation}
preserves positivity, and it is not difficult to show that it improves positivity (see Lemma~\ref{lem:R_0}).
In the formula~\eqref{eq:H lambda} for $H_P+\lambda$, this is perturbed by the operator $T_\lambda$. The off-diagonal part $T_{\mathrm{od},\lambda}$ is an integral operator with negative kernel, so $-T_{\mathrm{od},\lambda}$ preserves positivity. This can be dealt with by a perturbative argument due to Faris~\cite{faris1972}, the essential point of which is that the Neumann series $(1+A)^{-1}=\sum_{j=0}^\infty (-A)^j$ is positivity-preserving if $-A$ is. However, the diagonal part $T_{\mathrm{d},\lambda}$ is a multiplication operator by a function that does take positive values.
This is the main obstacle to turning this argument into a rigorous proof. It will be dealt with by slightly changing the representation of $H_P$, as explained in Section~\ref{sect:repr}.

\begin{rem}\label{rem:forms}
 If, in deviating from our hypothesis, 
\begin{equation}
 \int_{\R^3} \frac{|v(k)|^2}{k^2 + \omega(k)}\ud k < \infty,
\end{equation}
as is the case for the H. Fr\"ohlich's polaron model where $\omega\equiv 1$ and $v(k)\propto 1/k$, the formal Hamiltonian makes sense as a quadratic form. Self-adjointness can be proved as in~\cite[Sect.2]{LaSch19}, and the operator $T_\lambda$ is simply given by
\begin{equation}
 T_\lambda=-a(v)(L_P+\lambda)^{-1}a^*(v).
\end{equation}
Hence $-T_\lambda$ preserves positivity and the proof that the resolvent of $H_P$ improves positivity is rather straightforward starting from there.
\end{rem}

\section{Proof of positivity}
\subsection{A modified representation of $H_P$}\label{sect:repr}

We will deal with the positive part of $T_{\ud,\lambda}$ by absorbing it with $L_P$ and modifying the representation of $H_P$. A similar idea was used in the renormalisation of more singular Hamiltonians of Nelson type~\cite{La19, La20}.

For arbitrary $n\in \N_0$ (which we suppress in the notation) and $K\in \R^{3n}$ let 
\begin{equation}
\tau_{+, \lambda}(K):= (T_{\ud, \lambda}(K))_+
\end{equation}
be the positive part of the function $T_{\ud, \lambda}(K)$ given in~\eqref{eq:T_d def}.
By scaling (see Lemma~\ref{lem:T_d} for details) we have 
\begin{equation}\label{eq:tau-bound}
\tau_{+,\lambda}(K) \leq C (L_P(K)+\lambda)^{\eps} 
\end{equation}
for any $\eps>0$ and some $C$. Thus for every $\lambda>0$ and $\eps>0$, $\tau_{+,\lambda}(K)$ defines a bounded operator from $D(L_P^{\eps})$ to $\hilb_P$. We denote this operator by $\tau_{+,\lambda}$ and define $\tau_{-,\lambda}$ as 
\begin{equation}
\tau_{-,\lambda}:=T_{\ud,\lambda} - \tau_{+,\lambda}\leq 0.
\end{equation}
For  $\lambda > 0$ we now define $F_\lambda$, a modification of $G_\lambda$, as the adjoint of
\begin{equation}
  F_\lambda^*=-a(v) (L_P + \tau_{+,\lambda} + \lambda)^{-1}.
\end{equation}

\begin{lem}\label{lem:F}
The family of operators $F_\lambda$ has the following properties:
\begin{enumerate}[label=\alph*)]
 \item\label{lem:F B} $F_\lambda$ is bounded;
 \item\label{lem:F bound} $\ran F_\lambda \subset D(L_P^s)$ for all $0\leq s<1/4$ and for all $\lambda_0>0$
 \begin{equation*}
  \sup_{\lambda \geq \lambda_0} \|(L_P+\lambda)^s F_\lambda\|_{\mathscr{B}(\hilb_P)} <\infty.
 \end{equation*}
 \item\label{lem:F O-bound}  $F_\lambda$ maps $D(\ud \Gamma(\omega)^{1/2})$ to itself for all $\lambda_0>0$ there exists $C>0$ so that for all $\lambda\geq \lambda_0$ and $\psi\in D(\ud \Gamma(\omega)^{1/2})$
\begin{equation*}
 \|\ud \Gamma(\omega)^{1/2} F_\lambda \psi\|_{\hilb_P} \leq C \lambda^{-1/4} \|\ud \Gamma(\omega)^{1/2} \psi\|_{\hilb_P};
\end{equation*}
 \item\label{lem:F inv} There exists $\lambda_0>0$ so that for $\lambda>\lambda_0$, $1-F_\lambda$ is invertible on $\hilb_P$ and $D(\ud \Gamma(\omega)^{1/2})$ with 
 \begin{equation*}
  \sup_{\lambda>\lambda_0} \Big( \|(1-F_\lambda)^{-1}\|_{\mathscr{B}(\hilb_P)} + \|(1-F_\lambda)^{-1}\|_{\mathscr{B}(D(\ud \Gamma(\omega)^{1/2}))}\Big)<\infty
 \end{equation*}
 \item\label{lem:F equiv} $D(H_P)=(1-F_\lambda)^{-1} D(L_P)$.
\end{enumerate}
\end{lem}
\begin{proof}
Statements~\ref{lem:F B}--\ref{lem:F O-bound} are proved by reduction to the corresponding properties of $G_\lambda$. By~\eqref{eq:tau-bound} we have for $\psi \in \hilb_P$
 \begin{equation}
  \norm{\tau_{+,\lambda}(L_P + \tau_{+, \lambda} + \lambda)^{-1} \psi}_{\hilb_P} \leq C \norm{ (L_P+\lambda)^{-1 + \eps}  \psi}_{\hilb_P}.
 \end{equation}
The resolvent formula gives
 \begin{align}
  F_\lambda^*=G_\lambda^*\Big(1-\tau_{+,\lambda}(L_P + \tau_{+,\lambda} + \lambda)^{-1}\Big),
 \end{align}
 which thus implies~\ref{lem:F B}, since $G_\lambda$ is bounded (see Lemma~\ref{lem:G}\ref{lem:G-a}). 
The difference of the adjoints is then
\begin{equation}\label{eq:GF diff}
 G_\lambda - F_\lambda=(L_P + \tau_{+,\lambda} + \lambda)^{-1}\tau_{+,\lambda} G_\lambda.
\end{equation}
By Lemma~\ref{lem:G}\ref{lem:G L-bound}, $(L_P+\lambda)^{s}G_\lambda$ is bounded, uniformly in $\lambda$.
Thus by taking $\eps=s<1/4$ in~\eqref{eq:tau-bound}, $\tau_{+,\lambda} G_\lambda$ is bounded.
This implies ~\ref{lem:F bound} and~\ref{lem:F O-bound} since $L_P+\tau_+\geq L_P$ and thus $(L_P+\lambda)(G_\lambda-F_\lambda)$ is a bounded operator, uniformly in $\lambda$.

Statement~\ref{lem:F inv} follows from~\ref{lem:F bound} and~\ref{lem:F O-bound} by Neumann series, as in Lemma~\ref{lem:G}\ref{lem:G inv}.

Finally,~\ref{lem:F equiv} follows from the fact that $\ran \left(G_\lambda - F_\lambda\right)\subset D(L_P)$ as proved by~\eqref{eq:GF diff}.
 \end{proof}

\begin{prop}\label{prop:H F}
For every $\lambda > 0$ we have the identity
 \begin{equation*}
H_P=(1-F_\lambda)^*(L_P+ \tau_{+,\lambda} +\lambda)  (1-F_\lambda) + S_\lambda-\lambda
 \end{equation*}
with $S_\lambda = S_{\ud, \lambda} + S_{\uod,\lambda}$, $D(S_\lambda)=D(T_\lambda)$, given by
\begin{align*}
&\left(S_{\ud, \lambda}\psi\uppar{n}\right)(K) \\
&= \left(\tau_{-,\lambda} (K) + \int_{\R^3}  \frac{|v(\xi)|^2 \tau_{+,\lambda} (K, \xi)}{(L_P(K,\xi)+\lambda)(L_P(K,\xi) + \tau_{+,\lambda}(K, \xi)+\lambda)} \ud \xi\right)\psi\uppar{n}(K)
\end{align*}
and (for $n>0$, while $S_{\uod,\lambda}\psi\uppar{0}=0$) 
\begin{align*}
 \left(S_{\uod,\lambda}\psi\uppar{n}\right)(K)
 = - \sum_{j=1}^{n} \int_{\R^3}  \frac{\overline{v(\xi)} v(k_{j}) \psi \uppar n(\hat{K}_j,\xi )}{L_P(K,\xi) + \tau_{+,\lambda}(K, \xi) + \lambda} \, \ud \xi.
\end{align*}
\end{prop}
\begin{proof}
 We can rewrite the expression for $S_{\uod, \lambda}$ as 
 \begin{align}
   &\left(S_{\uod,\lambda}\psi\uppar{n}\right)(K) \notag \\
   &=  \left(T_{\uod,\lambda}\psi\uppar{n}\right)(K) + \sum_{j=1}^{n} \int_{\R^3}  \frac{\overline{v(\xi)} v(k_{j})\tau_{+,\lambda}(K, \xi) \psi \uppar n(\hat{K}_j,\xi )}{(L_P(K,\xi)+\lambda)(L_P(K,\xi) + \tau_{+,\lambda}(K, \xi)+\lambda)}.
 \end{align}
Putting the second term together with the second term of $S_{\ud, \lambda}$ yields the identity 
 \begin{align}
  S_\lambda = & T_{\uod,\lambda}+ \tau_{-,\lambda} + a(v)(L_P + \tau_{+,\lambda} + \lambda)^{-1}\tau_{+,\lambda}(L_P + \lambda)^{-1}a^*(v)   \notag \\
  = &   T_{\uod, \lambda} + \tau_{-,\lambda} + F_\lambda^*\tau_{+,\lambda} G_\lambda,
\end{align}
where the last expression is a bounded operator on $\hilb_P$ by the proof of Lemma~\ref{lem:F}.
Now using this, we obtain
\begin{align}
 H_P= &   (1-G_\lambda)^*(L_P+\lambda) (1-G_\lambda)+ T_\lambda - \lambda \notag \\
 = &   (1-G_\lambda)^*(L_P+\lambda) (1-G_\lambda) +\tau_{+,\lambda}  + S_\lambda - F_{\lambda}^*\tau_{+,\lambda} G_\lambda  -\lambda\notag\\
 = &  (1-G_\lambda)^*(L_P+\tau_{+,\lambda} + \lambda) (1-G_\lambda)  +  S_\lambda - \lambda \notag \\
 &+(1-G_\lambda^*)\tau_{+,\lambda}G_\lambda + G^*_\lambda \tau_{+,\lambda} - F_{\lambda}^*\tau_{+,\lambda} G_\lambda .
\end{align}
By~\eqref{eq:GF diff} we have $\tau_{+,\lambda}G_\lambda=(L_P+\tau_{+,\lambda}+\lambda) (G_\lambda-F_\lambda)$, so we can rewrite the last line as
\begin{align}
 &G_\lambda^*\tau_{+,\lambda} + (1-G^*_\lambda) \tau_{+,\lambda} G_\lambda- F_{\lambda}^*\tau_{+,\lambda} G_\lambda \notag\\
 & = G_\lambda^*\tau_{+,\lambda}(1-F_\lambda) + (1-F_\lambda^*) \tau_{+,\lambda} G_\lambda - G_\lambda^*\tau_{+,\lambda} (G_\lambda - F_\lambda) \notag \\
  &= (G_\lambda - F_\lambda)^*(L_P+\tau_{+,\lambda} + \lambda)(1-F_\lambda) + (1-F_\lambda^*)(L_P+\tau_{+,\lambda} + \lambda)(G_\lambda-F_\lambda)  \notag \\
  &\qquad -( G_\lambda - F_\lambda)^*(L_P+\tau_{+,\lambda} + \lambda)(G_\lambda-F_\lambda) \notag \\
  &= (1-F_\lambda^*)(L_P+\tau_{+,\lambda} + \lambda)(1-F_\lambda) - (1-G_\lambda)^*(L_P+\tau_{+,\lambda} + \lambda) (1-G_\lambda).
\end{align}
This proves the identity as claimed.
\end{proof}

\subsection{Proof of Theorem~\ref{thm:H_P pos}} \label{sect:pos}


We start by proving that the principal part of $H_P$, as given in Proposition~\ref{prop:H F}, improves positivity.

\begin{lem}\label{lem:R_0}
Assume that $g<0$ and let $\lambda_0$ be as in Lemma~\ref{lem:F}\ref{lem:F inv}. Then, for all $\lambda> \lambda_0$
the operator
\begin{align}
 R_0(\lambda):= &  \Big((1-F_\lambda)^*(L_P+ \tau_{+,\lambda} +\lambda)  (1-F_\lambda)\Big)^{-1} \notag \\
 = &(1-F_\lambda)^{-1} (L_P+ \tau_{+,\lambda}+\lambda)^{-1} (1-F_\lambda^*)^{-1} \notag
\end{align}
is positivity-improving with respect to $\mathcal{C}_+$.
\end{lem}
\begin{proof}
 We need to show that for all non-zero $\phi, \psi\geq 0$ we have
 \begin{align}\label{eq:res pos1}
  \left \langle (1-F_\lambda^*)^{-1}\phi, (L_P+ \tau_{+,\lambda}+\lambda)^{-1}(1-F_\lambda^*)^{-1} \psi \right\rangle >0.
 \end{align}
Let $n\geq 0$ be such that $\phi\uppar{n+1}\neq0$. From the formula 
\begin{equation}
 F_\lambda^*\phi\uppar{n+1} (K)= \sqrt{n+1}\int_{\R^3} \frac{-v(\xi)}{L_P(K, \xi) + \tau_{+,\lambda}(K,\xi) + \lambda} \phi\uppar{n+1}(K,\xi)\ud \xi
\end{equation}
we see that $F_\lambda^*\phi\uppar{n+1} \neq0$, since the first factor of the integrand is strictly positive. By induction then $ \left(F_\lambda^*\right)^{n+1}\phi\uppar{n+1}\neq 0$. Since $(1-F_\lambda^*)^{-1} = \sum_{j=1}^\infty \left(F_\lambda^*\right)^{j}$ (by choice of $\lambda_0$), this implies that
\begin{equation}
 \left \langle (1-F_\lambda^*)^{-1}\phi,\varnothing \right \rangle >0,
\end{equation}
where $\varnothing \in \hilb_P$ denotes the vacuum vector.
This proves~\eqref{eq:res pos1}, since the same applies to $(1-F_\lambda^*)^{-1}\psi$ and the restriction of $(L_P+ \tau_{+,\lambda} +\lambda)^{-1}$ to the vacuum sector is a strictly positive number, so both arguments of the scalar product have non-zero overlap with the vacuum.
\end{proof}

To complete the proof of Theorem~\ref{thm:H_P pos}, we will use a perturbative argument for $S_\lambda$. Since $S_{\uod, \lambda}$ is an  integral operator with negative kernel, the key point is that now  $S_{\ud , \lambda}$ is essentially negative.

\begin{lem}\label{lem:S neg}
For all $\lambda_0>0$ there exists $\mu_0$ such that for $\mu>\mu_0$ and every $\lambda\geq \lambda_0$ the operator $-(S_\lambda - \mu)$ is positivity-preserving.
\end{lem}
\begin{proof}
 The operator $S_{\uod, \lambda}$ is an integral operator with negative kernel, hence $-S_{\uod, \lambda}$ is positivity-preserving. Further, on the $n$-particle sector, the operator $S_{\ud, \lambda}$ is a multiplication operator by the sum of $\tau_{-,\lambda}(K)$, which is non-positive, and a non-negative function.
 We will show that this function is bounded (uniformly in $\lambda\geq \lambda_0$ and $n$) and then choose 
 \begin{equation}
\mu_0 = \sup_{\lambda\geq \lambda_0,n\in \N_0 , K\in \R^{dn}} S_{\ud, \lambda}(K),
 \end{equation}
whence $-S_{\ud, \lambda}(K) + \mu\geq 0$ for almost every $K\in \R^{3n}$.
For this, first use~\eqref{eq:tau-bound} (with $\eps<1$) and then the Hardy-Littlewood rearrangement inequality, to obtain  for $\lambda \geq \lambda_0$
\begin{align}
 S_{\ud,\lambda}(K) - \tau_{-,\lambda} 
 = &  \int_{\R^3}  \frac{|v(\xi)|^2 \tau_{+,\lambda} (K, \xi)}{(L_P(K,\xi)+\lambda)(L_P(K,\xi) + \tau_{+,\lambda}(K, \xi)+\lambda)} \ud \xi \notag\\
 \leq & C \int_{\R^3}  \frac{1}{\omega(\xi)(L_P(K,\xi)+\lambda)^{2-\eps}} \ud \xi \notag\\
 %
 %
  \leq  & C\int_{\R^3}  \frac{1}{|\xi|(\xi^2 + \lambda_0)^{2-\eps}} \ud \xi,\label{eq:S-tau}
\end{align}
which yields the claim.
\end{proof}

\begin{proof}[Proof of Theorem~\ref{thm:H_P pos}]
 We start by deriving a formula for the resolvent of $H_P$ for sufficiently large $\lambda$. Let $\lambda_0$ be as in Lemma~\ref{lem:F}\ref{lem:F inv}, $\mu>\mu_0$ as in Lemma~\ref{lem:S neg}, $\lambda\geq\max\{\mu,\lambda_0\}$, and $R_0(\lambda)$ as in Lemma~\ref{lem:R_0}. Then, by the representation of Proposition~\ref{prop:H F}, we have
 \begin{align}
  (H_P+\lambda - \mu)R_0(\lambda)=1+(S_\lambda - \mu)R_0(\lambda).
 \end{align}
 We now prove that this is invertible by a Neumann series for sufficiently large $\lambda$.

By Lemma~\ref{lem:T_d} and Equation~\eqref{eq:S-tau} there exists a constant $C>0$ such that
\begin{align}
 &\|(S_{\ud, \lambda} -\mu) R_0(\lambda)\|_{\mathscr{B}(\hilb_P)}   \notag\\
 &\leq \|\tau_{-,\lambda}R_0(\lambda)\|_{\mathscr{B}(\hilb_P)}  + \|(S_{\ud, \lambda}-\tau_{-,\lambda}-\mu) R_0(\lambda)\|_{\mathscr{B}(\hilb_P)} \notag\\
 &\leq C \| (L_P+\lambda)^\eps R_0(\lambda) \|_{\mathscr{B}(\hilb_P)} + (\mu_0+\mu) \|R_0(\lambda)\|_{\mathscr{B}(\hilb_P)}.\label{eq:S_d rel bound}
\end{align}
Using that $(1-F_\lambda)^{-1}=1+F_\lambda(1-F_\lambda)^{-1}$, we obtain for $\eps<1/4$ by Lemma~\ref{lem:F}
\begin{align}
 \| (L_P+\lambda)^\eps R_0(\lambda) \|_{\mathscr{B}(\hilb_P)}
 \leq& \|(1-F_\lambda^*)^{-1}\|_{\mathscr{B}(\hilb_P)} \Big(\|(L_P+\lambda)^{-1+\eps}\|_{\mathscr{B}(\hilb_P)}\notag \\
 & + \| (L_P+\lambda)^{\eps}F_\lambda(1-F_\lambda)^{-1} (L_P+\lambda)^{-1}\|_{\mathscr{B}(\hilb_P)}\Big)\notag \\
  \leq& C \lambda^{\eps-1}.
\end{align}
In view of~\eqref{eq:S_d rel bound} we thus have
\begin{equation}
 \|(S_{\ud, \lambda} -\mu) R_0(\lambda)\|_{\mathscr{B}(\hilb_P)} \leq C \lambda^{\eps-1}.
\end{equation}
%
%
%

By Lemma~\ref{lem:S_od} we have for the off-diagonal part
\begin{equation}
 \|S_{\uod, \lambda}  R_0(\lambda)\|_{\mathscr{B}(\hilb_P)} \leq C \| \ud \Gamma(\omega)^{1/2} R_0(\lambda) \|_{\mathscr{B}(\hilb_P)}.
\end{equation}
Using Lemma~\ref{lem:F}\ref{lem:F inv} and the fact that $\|(L_P+\lambda)^{-1}\|_{\mathscr{B}(\hilb_P, D(\ud \Gamma(\omega)^{1/2}))}$ is bounded by a constant times $\lambda^{-1/2}$ we obtain
\begin{align}
 &\| \ud \Gamma(\omega)^{1/2} R_0(\lambda) \|_{\mathscr{B}(\hilb_P)} \notag\\
 &\leq  \|(1-F_\lambda)^{-1}\|_{\mathscr{B}(D(\ud \Gamma(\omega)^{1/2}))} \| (L_P+\lambda)^{-1}\|_{\mathscr{B}(\hilb_P, D(\ud \Gamma(\omega)^{1/2}))} \|(1-F_\lambda^*)^{-1}\|_{\mathscr{B}(\hilb_P)} \notag \\
 &\leq C \lambda^{-1/2}.
\end{align}

Altogether, we find that there exists $C>0$ such that
\begin{equation}
 \|(S_\lambda - \mu)R_0(\lambda)\| < C \lambda^{-1/2},
\end{equation}
and it follows that for large enough $\lambda$
\begin{align}
 (H_P+\lambda - \mu)^{-1}&=
 R_0(\lambda)\big(1+(S_\lambda - \mu)R_0(\lambda)\big)^{-1} \notag \\
 &= R_0(\lambda) \sum_{j=0}^\infty \Big( -(S_\lambda - \mu)R_0(\lambda) \Big)^j.
\end{align}
By Lemma~\ref{lem:S neg}, the sum defines a positivity-preserving operator, which is one-to-one since it is invertible. As $R_0(\lambda)$ improves positivity by Lemma~\ref{lem:R_0}, we have for every $\psi\in \mathcal{C}_+\setminus\{0\}$
\begin{equation}
 R_0(\lambda) \underbrace{\big(1+(S_\lambda - \mu)R_0(\lambda)\big)^{-1} \psi}_{\in \mathcal{C}_+ \setminus\{0\}} >0.
\end{equation}
This proves the claim for all $\lambda>\lambda_1$, for some $\lambda_1>0$. The property extends to all $\lambda>-\inf \sigma(H_P)$ since, for $\gamma>\lambda_1\geq \lambda>-\inf \sigma(H_P)$, $(\gamma-\lambda)(H_P+\gamma)^{-1}$ has norm less than one, and thus
\begin{align}
 (H_P+\lambda)^{-1} &= (H_P+\gamma)^{-1}\Big(1-(\gamma-\lambda)(H_P+\gamma)^{-1}\Big)^{-1}  \notag \\
 &= (H_P+\gamma)^{-1} \sum_{j=0}^\infty \Big((\gamma-\lambda)(H_P+\gamma)^{-1}\Big)^{n} 
\end{align}
improves positivity.
\end{proof}

\appendix

\section{Technical Lemmas}\label{app:ibc}

Here we reprove the key Lemmas of~\cite{LaSch19, IBCmassless} for the special case of the (massive or massless) Nelson model at fixed momentum.

\begin{lem}\label{lem:G}
The family of operators $G_\lambda$ has the following properties: 
\begin{enumerate}[label=\alph*)]
 \item For every $\lambda> 0$, the operator $G_\lambda$ is bounded;\label{lem:G-a}
 \item\label{lem:G L-bound} $\ran G_\lambda \subset D(L_P^s)$ for any $0\leq s<1/4$, and for all $\lambda_0>0$
\begin{equation*}
 \sup_{\lambda\geq  \lambda_0} \|(L_P+\lambda)^s G_\lambda \|_{\mathscr{B}(\hilb_P)} <\infty;
\end{equation*}
\item\label{lem:G O-bound} $G_\lambda$ maps $D(\ud \Gamma(\omega)^{1/2})$ to itself and for all $\lambda_0>0$ there exists $C>0$ so that for all $\lambda\geq\lambda_0$ and $\psi\in D(\ud \Gamma(\omega)^{1/2})$
\begin{equation*}
 \|\ud \Gamma(\omega)^{1/2} G_\lambda \psi\|_{\hilb_P} \leq C \lambda^{-1/4} \|\ud \Gamma(\omega)^{1/2} \psi\|_{\hilb_P};
\end{equation*}

\item There exists $\lambda_0$ so that for all $\lambda>\lambda_0$ the operator $1-G_\lambda $ is boundedly invertible on $\hilb_P$ and $D(\ud \Gamma(\omega)^{1/2})$, and
\begin{equation*}
 \sup_{\lambda>\lambda_0} \Big(\|(1-G_\lambda)^{-1}\|_{\mathscr{B}(\hilb_P)} + \|(1-G_\lambda)^{-1}\|_{\mathscr{B}(D(\ud \Gamma(\omega)^{1/2})} \Big)<\infty.
 \end{equation*}

\label{lem:G inv}
\end{enumerate}
\end{lem}

\begin{proof}
 For \ref{lem:G-a} and \ref{lem:G L-bound} it is sufficient to prove that 
 \begin{equation}
 -a(v)(L_P+\lambda)^{s-1}
 \end{equation}
defines a bounded operator on $\hilb_P$, uniformly in $\lambda$.
To prove this, we insert a factor of $\sqrt{\omega(\xi)/\omega(\eta)}$ and its inverse, and then use that $a b\leq (a^2 + b^2)/2$ as well as the symmetry in $\xi, \eta$, to obtain for $n\geq 1$
\begin{align}
 &\|a(v)(L_P+\lambda)^{s-1}\psi\uppar{n}\|^2_{\hilb_P\uppar{n-1}} \notag\\
 &= n \int\limits_{\R^{3(n-1)}} \ud Q \int\limits_{\R^3} \ud  \xi \int\limits_{\R^3} \ud \eta \frac{v(\eta)\psi\uppar{n}(Q,\xi) \omega(\xi)^{1/2}}{(L_P(Q,\eta)+\lambda)^{1-s}\omega(\eta)^{1/2}} \frac{v(\xi) \overline{\psi}\uppar{n}(Q,\eta)\omega(\eta)^{1/2}}{(L_P(Q,\xi)+\lambda)^{1-s}\omega(\xi)^{1/2}} \notag\\
 &\leq n \int\limits_{\R^{3(n-1)}} \ud Q \int\limits_{\R^3} \ud  \xi  \frac{\omega(\xi) |\psi\uppar{n}(Q,\xi)|^2}{L_P(Q,\xi)+\lambda} \int\limits_{\R^3}   \ud \eta \frac{|v(\eta)|^2 }{(L_P(Q,\eta)+\lambda)^{1-2s}\omega(\eta)}.
 \label{eq:G expand}
\end{align}
By the Hardy-Littlewood rearrangement inequality we have for $s<1/4$
\begin{align}
 \int_{\R^3}    \frac{|v(\eta)|^2 \ud \eta}{(L_P(Q,\eta)+\lambda)^{1-2s} \omega(\eta)}
 &\leq g^2 \int_{\R^3}    \frac{\ud \eta }{(\eta^2 + m^2)((P-\eta-\sum_{j=1}^{n-1}k_j)^2+\lambda)^{1-2s}} \notag \\
 & \leq  g^2 \int_{\R^3}  \frac{\ud \tau }{|\tau|^2 (\tau^2+ \lambda)^{1-2s}},
\end{align}
which is uniformly bounded for $\lambda\geq \lambda_0$.
Together with~\eqref{eq:G expand} and the symmetry of $\psi\uppar{n}$ this gives
\begin{align}
 \|a(v)(L_P+\lambda)^{s-1}\psi\uppar{n}\|^2_{\hilb_P\uppar{n-1}}  &
 \leq  n  C\int\limits_{\R^{3(n-1)}} \ud Q \int\limits_{\R^3} \ud  \xi  \frac{\omega(\xi) |\psi\uppar{n}(Q,\xi)|^2}{L_P(Q,\xi)+\lambda}
 \notag \\
 & = C \int\limits_{\R^{3n}} \ud K \frac{ \sum_{j=1}^n \omega(k_j) |\psi\uppar{n}(K)|^2}{L_P(K)+\lambda} \notag\\
 &\leq C \| \psi\uppar{n} \|_{\hilb_P\uppar{n}}^2.\label{eq:G sym}
\end{align}

To prove~\ref{lem:G O-bound}, we proceed as in~\eqref{eq:G expand} to obtain (denoting $\Omega(K)=\sum_{j=1}^n \omega(k_j)$)
\begin{align}
 &\|(1+\ud\Gamma(\omega))^{-1/2} a(v)(L_P+\lambda)^{-1}\psi\uppar{n}\|^2_{\hilb_P\uppar{n-1}} \notag\\
 &\leq 2n \int\limits_{\R^{3(n-1)}} \ud Q \int\limits_{\R^3} \ud  \xi \frac{ \omega(\xi) |\psi\uppar{n}(Q,\xi)|^2}{(1+\Omega(Q))\Omega(Q,\xi)^2} \int\limits_{\R^3}   \ud \eta \frac{\Omega(Q, \eta)^2 |v(\eta)|^2 }{(L_P(Q,\eta)+\lambda)^2\omega(\eta)}.
\end{align}
Spelling out $\Omega(Q, \eta)^2 = \Omega(Q)^2 + 2\omega(\eta)\Omega(Q)+ \omega(\eta)^2$, we have three terms to deal with.
The term with $\omega(\eta)^2$ leads to, 
 \begin{equation}
 \frac{1}{1+\Omega(Q)} \int_{\R^3}    \frac{\omega(\eta) |v(\eta)|^2 \ud \eta}{(L_P(Q,\eta)+\lambda)^2} \leq  \frac{g^2}{(\Omega(Q)+\lambda)^{1/2}}\int_{\R^3} \frac{\ud \tau}{(\tau^2 +1)^2},
 \end{equation}
by scaling and rearrangement.
Similarly, we find for the other terms
\begin{equation}
 \frac{2\Omega(Q)}{1+\Omega(Q)}\int_{\R^3}    \frac{|v(\eta)|^2 \ud \eta}{(L_P(Q,\eta)+\lambda)^2} \leq  \frac{2 g^2}{\Omega(Q)+\lambda}\int_{\R^3} \frac{\ud \tau}{|\tau| (\tau^2 +1)^2},
\end{equation}
and
\begin{equation}
  \frac{\Omega(Q)^2}{1+\Omega(Q)} \int_{\R^3}    \frac{|v(\eta)|^2 \ud \eta}{(L_P(Q,\eta)+\lambda)^2\omega(\eta)} \leq \frac{g^2}{(\Omega(Q)+\lambda)^{1/2}}\int_{\R^3} \frac{\ud \tau}{|\tau|^2 (\tau^2 +1)^2}.
\end{equation}
Using symmetry as in~\eqref{eq:G sym} this implies for $\lambda \geq \lambda_0$
\begin{align}
 \|(1+\ud\Gamma(\omega))^{-1/2}a(v)(L_P+\lambda)^{-1}\psi\uppar{n}\|^2_{\hilb_P\uppar{n-1}} \leq C \lambda^{-1/2} \| \ud\Gamma(\omega)^{-1/2}\psi\uppar{n}\|^2_{\hilb_P\uppar{n}},
\end{align}
and proves~\ref{lem:G O-bound} by duality.

To prove~\ref{lem:G inv} observe that~\ref{lem:G L-bound} and~\ref{lem:G O-bound} imply that 
\begin{equation}
 \|G_\lambda\|_{\mathscr{B}(\hilb_P)} + \|G_\lambda\|_{\mathscr{B}(D(\ud \Gamma(\omega)^{1/2}))} \leq C \lambda^{-s}
\end{equation}
for $s<1/4$. Thus for large enough $\lambda$ the inverse of $1-G_\lambda$ in both spaces exists and is given by the Neumann series, whose norm is bounded by $(1-C\lambda_0^{-s})^{-1}$.
\end{proof}

\begin{lem}\label{lem:T_d}
For any $\eps>0$ there exists $C>0$ such that for all $\lambda>0$, $n\in \N_0$ and $K\in \R^{3n}$
 \begin{equation*}
  |T_{\ud,\lambda}(K)| \leq C (L_P(K) + \lambda)^{\eps}.
 \end{equation*}
\end{lem}
\begin{proof}
We treat only the case $n> 0$, the case $n=0$ is similar  but simpler. We have (with $\Omega(K)=\sum_{j=1}^n \omega(k_j)$)
 \begin{align}
  |T_{\ud,\lambda}(K)|
  = & \left|\int_{\R^3} |v(\xi)|^2 \frac{(P-\sum_{j=1}^nk_j - \xi)^2 - \xi^2 + \Omega(K) + \lambda}{(\xi^2 +\omega(\xi))(L_P(K, \xi)+ \lambda)}  \ud \xi\right| \notag \label{eq:T_d1} \\
  \leq & \int_{\R^3} |v(\xi)|^2 \frac{(P-\sum_{j=1}^nk_j)^2 + 2 |\xi||P-\sum_{j=1}^nk_j|}{(\xi^2 +\omega(\xi))(L_P(K, \xi)+ \lambda)}\ud \xi \\
  &+ \int_{\R^3} |v(\xi)|^2 \frac{\Omega(K) + \lambda}{(\xi^2 +\omega(\xi))(L_P(K, \xi)+ \lambda)}   \ud \xi . \label{eq:T_d2}
 \end{align}
 To simplify the notation, we set $p:=P-\sum_{j=1}^nk_j$.
 The first term in~\eqref{eq:T_d1} and the term~\eqref{eq:T_d2} are bounded by almost identical arguments, so we only give the details for one of them.

 Using that $(\xi^2 + \omega(\xi))\geq \xi^{2-4\eps}\omega(\xi)^{2\eps}\geq |\xi|^{2-2\eps}$ (for $1\geq\eps>0$) in~\eqref{eq:T_d2}, to avoid a singularity in $\xi=0$, and then scaling out $\sqrt{\Omega(K)+\lambda}$, we obtain by rearrangement
\begin{align}
 \eqref{eq:T_d2} \leq &  (\Omega(K) + \lambda)^{\eps} g^2\int\limits_{\R^3} \frac{\ud \xi}{\xi^{3-2\eps}\Big(\Big(\frac{p}{\sqrt{\Omega(K)+\lambda}}-\xi\Big)^2 +1\Big)} \notag \\
 \leq &  (\Omega(K) + \lambda)^{\eps} g^2\int_{\R^3} \frac{\ud \xi}{\xi^{3-2\eps}(\xi^2 +1)}. 
\end{align}

 For the second term in~\eqref{eq:T_d1} we have
\begin{align}
& \int_{\R^3} |v(\xi)|^2 \frac{2 |\xi||p|}{(\xi^2 +\omega(\xi))((p-\xi)^2 +\Omega(k) +\omega(\xi)+ \lambda)} \ud \xi\notag \\
&\leq g^2 \int_{\R^3} \frac{2  |p|}{\xi^2((p-\xi)^2 +\lambda)} \ud \xi.
\end{align}
Scaling by $|p|\neq 0$  then yields
\begin{equation}
 \int_{\R^3} \frac{2  |p|}{\xi^2((p-\xi)^2 + \lambda)} \ud \xi
 \leq \int_{\R^3}\frac{2}{\xi^2(\frac{p}{|p|}-\xi)^2}\ud \xi=C,
\end{equation}
where $C$ is independent of $p$ since the last integral is invariant by rotations.
 Combining these bounds proves the claim.
\end{proof}

\begin{lem}[{cf.~\cite[Lem.3.8]{IBCrelat}}]\label{lem:T_od}
There is $C>0$ so that  the inequality
\begin{equation*}
 \| T_{\uod, \lambda} \psi\uppar{n} \|_{\hilb_P\uppar{n}} \leq C \| \ud \Gamma(\omega)^{1/2} \psi\uppar{n} \|_{\hilb_P\uppar{n}}
\end{equation*}
holds for all $\lambda>0$ and $n\in \N$.
\end{lem}
\begin{proof}
 We may write
 \begin{align}
  T_{\uod, \lambda} \psi\uppar{n}(K)= -\sum_{j=1}^n \int_{\R^3}\frac{v(k_j)\omega(\xi)^{1/2}\psi\uppar{n}(\hat K_j, \xi)}{(L_P(K, \xi) + \lambda)^{1/2} \omega(k_j)^{1/2}} \frac{v(\xi)\omega(k_j)^{1/2}\ud \xi}{(L_P(K, \xi) + \lambda)^{1/2} \omega(\xi)^{1/2}}.
 \end{align}
By the Cauchy-Schwarz inequality we obtain
\begin{align}
  |T_{\uod, \lambda} \psi\uppar{n}(K)|^2 
  \leq &\left( \sum_{j=1}^n \int_{\R^3}\frac{|v(k_j)|^2 \omega(\xi)|\psi\uppar{n}(\hat K_j, \xi)|^2 }{(L_P(K, \xi) + \lambda) \omega(k_j)}\ud \xi \right) \label{eq:T_od1}\\
  & \times \left(\sum_{\ell=1}^n \int_{\R^3}\frac{|v(\eta)|^2\omega(k_\ell)}{(L_P(K, \eta) + \lambda) \omega(\eta)}\ud \eta\right).\label{eq:T_od2}
\end{align}
The factor~\eqref{eq:T_od2} is bounded by (writing $\Omega(K)=\sum_{j=1}^n \omega(k_j)$ and $p=P-\sum_{j=1}^n k_j$)
\begin{align}
 &\sum_{\ell=1}^n \int_{\R^3}\frac{|v(\eta)|^2\omega(k_\ell)}{(L_P(K, \eta) + \lambda) \omega(\eta)}\ud \eta \notag\\
 %
 %
 &\leq g^2\sum_{\ell=1}^n \omega(k_\ell) (\Omega(K)+\lambda)^{-1/2} \int_{\R^3}\frac{1}{\Big(\Big(\frac{p}{\sqrt{\Omega(K)+\lambda}} - \eta\Big)^2 +1\Big) |\eta|^{2}}\ud \eta\notag \\
 &\leq g^2 \sum_{\ell=1}^n \omega(k_\ell) \Omega(K)^{-1/2} \int_{\R^3}\frac{1}{(\eta^2 +1) |\eta|^{2}}\ud \eta,
\end{align}
where in the final step we have used the Hardy-Littlewood inequality on the integral.

We use this to estimate the integral of $|T_{\uod, \lambda} \psi\uppar{n}(K)|^2$ over $K$ and obtain
\begin{align}
 & \| T_{\uod, \lambda} \psi\uppar{n} \|_{\hilb_P\uppar{n}}^2 \notag \\
 & \leq C   \sum_{j, \ell=1}^n\int_{\R^{3n}} \int_{\R^3}\frac{ \omega(k_\ell) \Omega(K)^{-1/2} |v(k_j)|^2 \omega(\xi)|\psi\uppar{n}(\hat K_j, \xi)|^2 }{(L_P(K, \xi) + \lambda) \omega(k_j)} \ud \xi \ud K.
\end{align}
By renaming the variables $k_j=\eta$, $\xi=k_j$ in the $j$-th integral, and using the symmetry of $\psi\uppar{n}$, this becomes
\begin{align}
  \| T_{\uod, \lambda}& \psi\uppar{n} \|_{\hilb_P\uppar{n}}^2 \notag \\
  \leq & \sum_{j\neq \ell=1}^n\int_{\R^{3n}} \int_{\R^3}\frac{\omega(k_\ell) \Omega(\hat K_j, \eta)^{-1/2} |v(\eta)|^2 \omega(k_j)|\psi\uppar{n}(K)|^2 }{(L_P(K, \eta) + \lambda) \omega(\eta)} \ud \eta \ud K \label{eq:Tod bound1}\\
  & + \sum_{j=1}^n\int_{\R^{3n}} \int_{\R^3}\frac{ \Omega(\hat K_j, \eta)^{-1/2} |v(\eta)|^2 \omega(k_j)|\psi\uppar{n}(K)|^2 }{(L_P(K, \eta) + \lambda) } \ud \eta \ud K. \label{eq:Tod bound2}
\end{align}
The first term is bounded by
\begin{align}
 \eqref{eq:Tod bound1} \leq &  \sum_{j=1}^n
 \int_{\R^{3n}}  \int_{\R^3}\frac{\Omega(\hat K_j)^{1/2} |v(\eta)|^2 \omega(k_j)|\psi\uppar{n}(K)|^2 }{(L_P(K, \eta) + \lambda) \omega(\eta)} \ud \eta \ud K \notag\\
 \leq &  C \int_{\R^{3n}} \sum_{j=1}^n \omega(k_j)\Omega(K)^{-1/2} \Omega(\hat K_j)^{1/2}|\psi\uppar{n}(K)|^2  \ud K \notag \\
 \leq & C  \int_{\R^{3n}} \Omega(K)|\psi\uppar{n}(K)|^2  \ud K,
\end{align}
by the same scaling argument as before.

By the same reasoning, the second term satisfies the bound
\begin{align}
 \eqref{eq:Tod bound2} \leq &\sum_{j=1}^n\int_{\R^{3n}} \int_{\R^3}\frac{ |v(\eta)|^2 \omega(k_j)|\psi\uppar{n}(K)|^2 }{(L_P(K, \eta) + \lambda) \omega(\eta)^{1/2} } \ud \eta \ud K \notag \\
 \leq & C  \int_{\R^{3n}} \Omega(K) |\psi\uppar{n}(K)|^2 \ud K
\end{align}
and this proves the claim.
\end{proof}

\begin{lem}\label{lem:S_od}
There is a $C>0$ so that  the inequality
\begin{equation*}
 \| S_{\uod, \lambda} \psi\uppar{n} \|_{\hilb_P\uppar{n}} \leq C  \| \ud \Gamma(\omega)^{1/2} \psi\uppar{n} \|_{\hilb_P\uppar{n}}
\end{equation*}
holds for all $\lambda>0$ and $n\in \N$.
\end{lem}
\begin{proof}
 As $\tau_{+,\lambda}(K)\geq 0$, the proof is identical to that of Lemma~\ref{lem:T_od}. 
\end{proof}

%

\end{document}